\documentclass[12pt]{amsart}

\usepackage{amsfonts,amssymb}
\usepackage{hyperref, graphicx}
\usepackage{epsfig}
\usepackage{latexsym}
\usepackage{amsmath,amsthm}
\usepackage{mathrsfs}
\usepackage[all,cmtip]{xy}
\theoremstyle{plain}
\newtheorem{theorem}{Theorem}[section]
\newtheorem{lemma}[theorem]{Lemma}
\newtheorem{definition}[theorem]{Definition}
\newtheorem{proposition}[theorem]{Proposition}

\newtheorem{remark}[theorem]{Remark}

\numberwithin{equation}{section}

\newcommand{\ra}{\longrightarrow}

\newcommand{\N}{\mathbb{N}}

\newcommand{\C}{\mathcal{C}}

\newcommand{\nil}[1]{{#1}^{\mathrm{uni}}}

\newcommand{\Efgs}[2]{\pi^{\rm \acute{e}t}(#1,\, #2)}
\newcommand{\Nfgs}[2]{\pi^N(#1,\, #2)}
\newcommand{\Sfgs}[2]{\pi^S(#1,\, #2)}
\newcommand{\Ffgs}[2]{\pi_F(#1,\, #2)}
\newcommand{\Ufgs}[2]{\pi^{\mathrm{uni}}(#1,\, #2)}
\newcommand{\Ffgsu}[2]{\pi_F^\mathrm{uni}(#1,\, #2)}
\newcommand{\Gfgp}[2]{\pi_1(#1,\, #2)}
\newcommand{\Lfgp}[2]{\pi^{\rm loc}(#1,\, #2)}

\newcommand{\Ns}[1]{\mathbf{Ns}(#1)}
\newcommand{\sheaf}[1]{\ensuremath{\mathcal{#1}}}
\newcommand{\cat}[1]{\ensuremath{\mathcal{#1}}}

\newcommand{\struct}[1]{\ensuremath{\mathcal{O}_{#1}}}

\begin{document}
\baselineskip=15.5pt

\title{A note on certain Tannakian group schemes}
\author{Sanjay~Amrutiya}
\address{Department of Mathematics, IIT Gandhinagar,
 Near Village Palaj, Gandhinagar - 382355, India}
 \email{samrutiya@iitgn.ac.in}
\date{}
\begin{abstract}
\noindent In this note, we prove that the $F$--fundamental group scheme is birational invariant for smooth projective varieties.
We prove that the $F$--fundamental
group scheme is naturally a quotient of the Nori fundamental group scheme. For elliptic curves, it turns out that the $F$-fundamental group 
scheme and the Nori fundamental group scheme coincides. We also consider an extension of the Nori fundamental
group scheme in positive characteristic using semi-essentially finite vector bundles and prove that in this way, 
we do not get a non-trivial extension of the Nori fundamental group scheme for elliptic curves, unlike in 
characteristic zero.
\end{abstract}
 \maketitle
\section{Introduction}
Madhav Nori (in \cite{No}) has introduced the fundamental group scheme of a reduced and connected scheme $X$ over 
a field $k$. If $X$ is a proper integral scheme over a field $k$ with $x\in X(k)\neq \emptyset$, then the the Nori 
fundamental group scheme, denote by $\Nfgs{X}{x}$, is Tannaka dual to a neutral Tannakian category of 
\emph{essentially finite} vector bundles over $X$. 
It was proved in \cite{No} that the fundamental group scheme is birational invariant for smooth projective varieties.
The notion of $S$-fundamental group scheme was introduced in \cite{BPS} for curves and in \cite{La11}
for higher dimensional varieties. The $S$-fundamental group scheme $\Sfgs{X}{x}$ for a pointed smooth projective 
variety is the affine group scheme corresponding to the Tannakian category of \emph{Nori-semistable} bundles on $X$. 
The birational invariance of the $S$-fundamental group scheme is proved in \cite{HM}.

In \cite{AB}, the authors constructed a neutral Tannakian catefory over a perfect field of prime characteristic for
any pointed  projective variety $(X, x)$. This neutral Tannakian category is a tensor generated category by the 
family of \emph{Frobenius-finite} vector bundles over $X$. The corresponding affine group scheme is called the 
$F$-fundamental group scheme of $X$, and is denoted by $\Ffgs{X}{x}$.

In this note, we prove that the $F$-fundamental group scheme is birational invariant following the techniques 
of \cite{HM}. We also observe that the $F$-fundamental group scheme is a quotient of the Nori fundamental group 
scheme, and for elliptic curves, they coincide using the results of Atiyah and Oda.  In \cite{Ot}, Otabe introduced 
an extension of the Nori fundamental group scheme using semi-finite vector bundles, and studied it's properties in 
characteristic zero. We  consider an analogue of this extension of the Nori fundamental group scheme in positive 
characteristic, and prove that we do not get a non-trivial extension of the Nori fundamental group scheme for 
elliptic curves. Over finite fields, all these Tannakian fundamental group schemes coincides.

\section{Preliminaries}
In this section, we recall definition of certain fundamental group scheme associated to a proper connected scheme 
defined over $k$. 

Let $X$ be a proper integral scheme defined over a perfect field $k$ endowed with a rational point $x\in X(k)$.
A vector bundle $E$ over $X$ is called \emph{\'{e}tale trivializable} if there exists a finite \'{e}tale covering 
$\psi : Y\ra X$ such that the pull-back $\psi^*E$ is trivializable.
The affine group scheme corresponding to the neutral Tannakian category $\cat{C}^{\rm \acute{e}t}(X)$ 
defined by all \'{e}tale trivializable vector bundles $E$ over $X$ (the fiber functor for this neutral Tannakian 
category sends $E$ to its fiber $E(x)$ over $x$) is called the \emph{(geometric) \'etale fundamental group scheme} 
and is denoted by $\Efgs{X}{x}$. When $k$ is algebraically closed field of characteristic zero, the algebraic fundamental 
group $\Gfgp{X}{x}$ is canonically isomorphic to the group of $k$-valued points of the \'etale fundamental group scheme 
$\Efgs{X}{x}$.

A vector bundle $E$ over $X$ is said to be \emph{Nori--semistable} if for every non-constant morphism
$f\,:\, C\,\longrightarrow\, X$ with $C$ a smooth projective curve, the pull-back $f^*E$ on $C$ 
is semi-stable of degree zero.

Let $\Ns{X}$ denote the category of all Nori--semistable vector bundles over $X$. With the usual notion of tensor
product and the fibre functor defined by the $k$-rational point $x$, the category $\Ns{X}$ defines a neutral Tannakian 
category over $k$. The affine group scheme corresponding to the neutral Tannakian category $\Ns{X}$ is called the 
$S$-funamental group scheme, and is denoted by $\Sfgs{X}{x}$ \cite{BPS, La11}.

For any vector bundle $E$ over $X$ and for any polynomial $f$ with
non-negative integer coefficients, we define
$$
f(E) := \bigoplus_{i=0}^n (E^{\otimes i})^{\oplus a_i},
$$
where $f(t) = \sum_{i=0}^n a_it^i$ with $a_i\in \N\cup \{0\}, \, \forall
\, i\in \{0, 1, 2,\dots, n\}$.

A vector bundle $E$ over $X$ is said to be \emph{finite} if there are distinct polynomials $f$ and $g$ with
non-negative integer coefficients such that $f(E)$ is isomorphic to $g(E)$.

For any vector bundle $E$, let $S(E)$ denote the collection of all the indecomposable components of $E^{\otimes n}$,
for all non-negative integers $n$. Then, the vector bundle $E$ is finite if and only if $S(E)$ is a finite set 
\cite[Lemma 3.1]{No}.
A vector bundle $E$ over $X$ is said to be \emph{essentially finite} if it is Nori-semistable subquotient
of a finite vector bundle over $X$.

Let $\cat{C}^N(X)$ be the full subcategory of $\Ns{X}$ of essentially finite vector bundles over $X$ 
(see \cite{No} for the definition of essentially finite vector bundles on $X$).
With usual tensor product of vector bundles and the fibre functor defined by the $k$-rational point $x$,  the category 
$\cat{C}^N(X)$ defines a neutral Tannakian category  over $k$. The affine group scheme corresponding to the 
neutral Tannakian category $\C^N(X)$ is called \emph{the Nori fundamental group scheme} of $X$ over $k$ with 
base point $x$, and is denoted by $\Nfgs{X}{x}$.

A vector bundle $E$ over $X$ is called $F$--trivial if $(F^i_X)^*E$ is trivial for some $i$.
Note that the category of all F--trivial vector bundles form a Tannakian category; the corresponding affine group 
scheme is called \emph{the local fundamental group scheme}, which is denoted by $\Lfgp{X}{x}$.

For any polynomial $g(t)\, =\, \sum_{i=0}^m n_i t^i$ with $a_i\in \N\cup \{0\}, \, \forall \, i\in \{0, 1,\dots, n\}$, define
$$
{\widetilde g}(E)\, :=\, \bigoplus_{i=0}^m ((F^i_X)^*E)^{\oplus n_i}\, ,
$$
where $F^0_X$ is the identity morphism of $X$.

A vector bundle $E$ over $X$ is called \textit{Frobenius--finite} if there are two distinct polynomials $f$ and $g$ 
of the above type such that ${\widetilde f}(E)$ is isomorphic to ${\widetilde g}(E)$.

For any vector bundle $E$ over $X$, let $\mathrm{I}(E)$ denote the set of all indecomposable components of 
$\{(F_X^n)^*E\}_{n\geq 0}.$

Let $\mathrm{TFF}(X)$ denote the collection of all finite tensor products of Frobenius--finite vector bundles over $X$. 
Consider the full subcategory, denoted by $\cat{C}_F(X)$, of the category $\Ns{X}$ whose objects are all Nori--semistable 
subquotients of finite direct sum of elements of $\mathrm{TFF}(X)$.

With the usual tensor product of vector bundles and the fibre functor defined by $x\in X(k)$, the category $\cat{C}_F(X)$ 
defines a neutral Tannakian category  over $k$.  The {\it $F$--fundamental group--scheme} of $X$ with the base point 
$x$, denoted by $\Ffgs{X}{x}$,  is the affine group scheme associated to the neutral Tannakian category
$(\cat{C}_F(X), T_x)$, where $T_x$ is the neutral fibre functor on $\C_F(X)$.

\section{Some properties of Tannakian group schemes}

We have the following diagram in which each arrow is a natural faithfully flat morphism of affine group schemes:
\[
\xymatrix{
\Sfgs{X}{x} \ar[r] \ar[rd] & \Nfgs{X}{x} \ar[r] \ar@{.>}[d] \ar[rd] & \Efgs{X}{x} \\
& \Ffgs{X}{x} \ar[r] & \Lfgp{X}{x} 
}
\]
The existence of the dotted arrow follows from Proposition \ref{Nori-F-relation}. For this we will use the following 
theorem of Lange and Stuhler.

\begin{proposition}\cite[Theorem 1.4]{LS}\label{LS-thm}
Let $E$ be a vector bundle on $X$ defined over $k$.
\begin{enumerate}
\item If $(F_X^n)^*E$ is isomorphic to $E$ for some $n > 0$, then $E$ is \'{e}tale trivializable.
\item If $E$ is stable \'{e}tale trivializable, then there exist some $n > 0$ such that $(F_X^n)^*E$ is isomorphic to $E$ (cf. \cite{BD}).
\item If $k = \mathbb{F}_p$ or $k = \overline{\mathbb{F}_p}$ and $E$ is \'{e}tale trivializable,
then there exist some $n > 0$ such that $(F_X^n)^*E$ is isomorphic to $E$.
\end{enumerate}
\end{proposition}

\begin{proposition}\label{Nori-F-relation}
There is a faithfully flat morphism $\Nfgs{X}{x}\ra \Ffgs{X}{x}$ of affine group schemes.
\end{proposition}
\begin{proof}
Let $E$ be Frobenius--finite vector bundle over $X$. Then $I(E)$ is a finite set. Since each Frobenius pullback
$(F_X^n)^*E$ is of rank $r = \mathrm{rk}(E)$, we can conclude that there are integers $m > k > 0$ such that
$(F_X^m)^*E \cong (F_X^k)^*E\,.$
Hence, we have $$(F_X^{m-k})^*(F_X^k)^*E = (F_X^m)^*E \cong (F_X^k)^*E\,.$$
By Proposition \ref{LS-thm}, there exists a finite \'etale covering $\psi \colon Y\ra X$ such that $\psi^*((F_X^k)^*E)$
is trivial. This proves that $E$ is essentially finite (see \cite[Chapter II, Proposition 7]{No} or \cite[Proposition 2.3]{BH07}). 
\end{proof}

\begin{proposition}\label{prop-FNS}
If $k = \mathbb{F}_p$, then $\Ffgs{X}{x} = \Nfgs{X}{x} = \Sfgs{X}{x}$.
\end{proposition}
\begin{proof}
We know that $\C^N(X) \subseteq \C_F(X) \subseteq \Ns{X}$. Therefore, it is enough to show that every 
Nori-semistable bundle on $X$ is essentially finite. To see this, let $E$ be a Nori-semistable vector bundle
on $X$. Then for each $n \geq 0$, the Frobenius pullbacks $(F_X^n)^*E$ is $H$-semistable and
$c_1(E)\cdot H^{n-1} = c_2(E)\cdot H^{n-2} = 0$ for any polarization $H$ on $X$. Hence, the family 
$\{(F_X^n)^*E\}_{n\geq 0}$ is bounded family, i.e., there exists a $k$-scheme $S$ of finite type and a coherent
sheaf $\sheaf{E}$ on $X\times S$ such that each $(F_X^n)^*E$ is isomorphic to $\sheaf{E}{|_{X\times \{s_n\}}}$ for
some $k$-point $s_n$ of $S$ (see \cite[Theorem 4.4]{La04}). Since $k$ is finite field, it follows that the family
$\{(F_X^n)^*E\}_{n\geq 0}$ contains only finitely many isomorphism classes. In other words, there are integers
$n_1 > n_0 >0$ such that $(F_X^{n_1})^*E \cong (F_X^{n_0})^*E$. Now, by Proposition \ref{LS-thm}, we conclude that 
certain Frobenius pullback of $E$ is \'etale trivializable, and hence $E$ is essentially finite vector bundle.
\end{proof}

\subsection*{Birational Invariance}
In \cite[Lemma 8.3]{La11}, Langer observed that $\Sfgs{X}{x}$ does not change while blowing up $X$ along smooth centres, 
and raised the question of birational invariance of the $S$-fundamental group scheme. For smooth projective surfaces, this would imply the birational invariance of the $S$-fundamental group scheme using weak factorization theorem. Due to lack of
weak factorization theorem for higher dimension in positive characteristic, this problem required a different techniques.
In \cite{HM}, Hogadi and Mehta proved the birational invariance of the $S$-fundamental group scheme using 
a vanishing theorem \cite[Theorem 2.1]{HM}.

In the same vein, we can see that the blow ups along smooth centres do not affect the $F$-fundamental group scheme.
To see this, let $\varphi \,:\, \mathrm{Bl}_Z(Y)\ra Y$ be a blow-up of $Y$ along a smooth center $Z$. 
By \cite[Proposition 3.2]{AB}, 
it follows that the induced homomorphism
$$
\hat{\varphi}: \Ffgs{\mathrm{Bl}_Z(Y)}{x}\ra \Ffgs{Y}{y}
$$
is faithfully flat. To see that $\hat{\varphi}$ is closed immersion, by \cite[Proposition 2.21]{DM}, we need to 
show that for $E\in \C_F(\mathrm{Bl}_Z(Y))$ there exists $E'\in \C_F(Y)$ such that $E$ is isomorphic to the 
sub--quotient of $\varphi^*E'$. By \cite[Theorem 1]{Ish}, one can conclude that $\varphi_*E$ is locally free and 
$E\simeq \varphi^*\varphi_*E$ (cf. \cite[Lemma 8.3]{La11}). Take $E' = \varphi_*E$. 
Note that the following diagram
\[
\xymatrix @C=3.5pc{
\mathrm{Bl}_Z(Y) \ar[r]^{F_{\mathrm{Bl}_Z(Y)}} \ar[d]_{\varphi} &
\mathrm{Bl}_Z(Y)\ar[d]^{\varphi}\\
Y \ar[r]_{F_Y} & Y
}
\]
is cartesian. From this, one can conclude that the vector bundle $E'$ is an object of $\C_F(Y)$.

In the following, we will see that using the techniques of \cite{HM} one can easily
deduced the birational invariance of the $F$-fundamental group scheme.

\begin{theorem}
Let $X$ and $Y$ be smooth projective $k$-varieties, and let $\phi \colon X\dashrightarrow Y$ be a birational map. 
Let $x_0\in X(k)$ be a point where $\phi$ is defined. Then $\Ffgs{X}{x}$ is isomorphic to $\Ffgs{Y}{\phi(x)}$.
\end{theorem}
\begin{proof}
Consider the diagram
\[
\xymatrix{
& Z \ar[ld]_p \ar[rd]^q &\\
X \ar@{.>}[rr]_\phi && Y
}
\]
where $Z$ is the normalization of the closure of graph of $\phi$, the morphisms $p$ and $q$ 
are birational. It is proved in \cite{HM} that if $E\in \Ns{X}$, then $q_*p^*E \in \Ns{Y}$, and hence the functor 
$$
q_*p^* \colon \Ns{X}\ra \Ns{Y}
$$
is an equivalence of Tannaka categories. To conclude that the $F$-fundamental group scheme is 
birational invarience, we have to check that the functor $q_*p^*$ induces an equivalence
between $\C_F(X)$ and $\C_F(Y)$. For this, let $E$ be a Frobenius--finite vector bundle on
$X$. Then, there exist two integers $m_1 > m_0 > 0$ such that we have
an isomorphism
\begin{equation}
(F_X^{m_1})^*E \cong (F_X^{m_0})^*E
\end{equation}
of vector bundles on $X$. Let $E' := p^*E$. Since $p\circ F_Z = F_X\circ p$, we have an isomorphism
\begin{equation}\label{eq-ffb}
(F_Z^{m_1})^*E' \cong (F_Z^{m_0})^*E'
\end{equation}
of vector bundles on $Z$. 
By \cite[Lemma 3.1]{HM}, we have an isomorphism 
\begin{equation}\label{base-change-isom}
(F_Y^n)^*(q_*E') \cong q_*((F_Z^n)^*E')
\end{equation}
of vector bundles on $Y$, for all $n \geq 0$. Using isomorphisms \eqref{base-change-isom} and \eqref{eq-ffb}, 
we can conclude that $q_*p^*E$ is a Frobenius--finite vector bundle on $Y$.

Let $E\in \C_F(X)$. Then there are finitely many elements $E_1, E_2, \dots, E_m $ in $\mathrm{TFF}(X)$ and Nori-semistable
subbundles
$$
V_1 \subseteq V_2 \subseteq \bigoplus_{i = 1}^m E_i
$$ 
such that $E\cong V_2/V_1$. Since the functor $q_*p^*$ takes Frobenius--finite vector bundles on $X$ to Frobenius--finite 
vector bundles on $Y$, it follows that $q_*p^* E_i \in \mathrm{TFF}(Y)$. It is easy to see that
$$
q_*p^* E \cong q_*p^*(V_2/V_1)\cong q_*p^*V_2/q_*p^*V_1\,.
$$
Note that $q_*p^*V_2$ and $q_*p^*V_1$ are Nori-semistable. This proves that $q_*p^* E \in \C_F(Y)$.

\end{proof}

\subsection*{Unipotent group scheme}
Recall that a vector bundle $E$ on $X$ is called unipotent if there exists a filtration
$$
0 = E_0\subset E_1\subset E_2\subset \cdots \subset E_n = E
$$
such that $E_i/E_{i-1} \cong \struct{X}$ for $1\leq i\leq n$.
In \cite[Chapter IV]{No}, Nori proved that the category $\nil{\C}(X)$ of unipotent vector bundles on $X$
is a Tannaka category. For $x\in X(k)$, the corresponding affine group scheme is called the unipotent group 
scheme of $X$ and is denoted by $\Ufgs{X}{x}$. 

If characteristic of $k$ is zero, then the category $\C^N(X)$ is semi-simple (cf. \cite[p. 84]{No}) and hence 
any finite nilpotent vector bundle will be trivial vector bundle. However, if characteristic of $k$ is $p > 0$, 
then there is a natural faithfully flat morphism $\Nfgs{X}{x}\ra \Ufgs{X}{x}$ (see \cite[Chapter IV; Proposition 3]{No}).

The largest unipotent quotient of the $F$-fundamental group scheme of $X$ is called the {\it unipotent part} of 
$\Ffgs{X}{x}$ and is denoted by $\Ffgsu{X}{x}$. By Tannaka duality, the category of finite dimensional 
$k$-representations of $\Ffgsu{X}{x}$ is equivalent (as a neutral Tannakian category) to the category $\nil{\C_F}(X)$. 
Moreover, we have a natural faithfully flat morphism $\Ufgs{X}{x}\ra \Ffgsu{X}{x}$ of affine group schemes.

If $k = \mathbb{F}_p$ or $k=\overline{\mathbb{F}_p}$, then by Proposition \ref{prop-FNS}, we have
 $\Ufgs{X}{x} = \Ffgsu{X}{x}$ (cf. \cite[Corollary 4.3]{AB}).

\subsection*{An extension of the Nori fundamental group scheme}
In \cite{Ot}, S. Otabe introduced an extension of the Nori fundamental group scheme using semi-finite vector bundles,
and studied it's properties in characteristic zero. 

To get the right analogue of an extension of the Nori fundamental group scheme in positive characteristic, we may 
consider the following:

\begin{definition}\cite[cf. Definition 2.2]{Ot}\rm{
A vector bundle $E$ on $X$ is called semi-essentially finite if there exists a filtration
$$
0 = E_0\subset E_1\subset E_2\subset \cdots \subset E_n = E
$$
such that $E_i/E_{i-1}$ is essentially finite indecomposable for all $i = 1, 2, \dots, n$.
}
\end{definition}

Let $\C^{EN}(X)$ be the category of semi-essentially finite vector bundles on $X$. 
By following the similar argument as in \cite[Chapter IV, Lemma 2]{No}, \cite[Propositiion 2.14]{Ot}, we have 

\begin{proposition}
With the usual tensor product of vector bundles and a neutral fibre functor $T_x$ of $\C^{EN}(X)$ which sends a 
vector bundle $E\in \C^{EN}(X)$ to its fibre over $x$, the category $\C^{EN}(X)$ defines a neutral Tannakian 
category over $k$.
\end{proposition}

The affine group scheme corresponding to $\C^{EN}(X)$ is denoted by $\pi^{EN}(X, x)$. From the definition, we have the following sequence of faithfully flat homomorphisms 
$$
\Sfgs{X}{x} \ra \pi^{EN}(X, x) \ra \Nfgs{X}{x} 
$$
of affine group schemes. For, let $E\in \C^N(X)$. Then there exists a finite principal $G$-bundle $\pi\colon P\ra X$
such that $\pi^*E$ is trivial. We have a decomposition of vector bundle $E$ into direct sum of indecomposable bundles.
Using induction, it is enough to assume that $E = E_1 \oplus E_2$. Since $\pi^*E$ 
is trivial, it follows that $\pi^* E_1$  and $\pi^* E_2$ are degree zero subbundle of trivial vector bundle. This implies that $E_1$ and $E_2$ are essentially finite vector bundles.

If $X$ is an elliptic curve defined over an algebraically closed field of characteristic zero, then it is known 
that $\pi^{EN}(X, x) \cong \Nfgs{X}{x} \times \Ufgs{X}{x}$ \cite[Theorem 3.10]{Ot}. The situation in positive 
characteristic is quite different (see Theorem \ref{thm-EN-Nori}). 

\section{Tannakian group schemes of elliptic curves}
Let $C$ be an elliptic curve defined over a perfect filed $k$ of positive characteristic. 
Using Atiyah's classification of indecomposable vector bundles on elliptic curves, we have
 
\begin{lemma}\label{lemma-2}
Let $E$ be a Nori-semistable vector bundle on $C$. Then 
\begin{equation}\label{eq-main}
E\cong \bigoplus_{i = 1}^k F_{r_i}\otimes L_i
\end{equation}
where each $L_i$ is a line bundle of degree $0$.
\end{lemma}
\begin{proof}
Since $E$ is Nori-semistable, we have 
$
E = \bigoplus_{i = 1}^k E_i\;,
$
where $E_i\in \mathcal{E}_C(r, 0)$. By \cite[Theorem 5]{At}, we have $E_i\cong F_{r_i}\otimes L_i$ where each $L_i$ is a line bundle of degree $0$.
\end{proof}

\begin{theorem}\label{thm-Ffg-to-Nori} 
Let $C$ be an elliptic curve defined over a perfect filed $k$ of positive characteristic $p$. For a $k$-rational point
$x\in C(k)$, we have $\Ffgs{C}{0} = \Nfgs{C}{0}$.
\end{theorem}

\begin{proof}
Let us first consider the case that $C$ is ordinary. Let $E$ be an essentially finite vector bundle on $C$.
Then by Lemma \ref{lemma-2} , it follows that 
$
E\cong \bigoplus_{i = 1}^k F_{r_i}\otimes L_i
$
where each $L_i$ is a line bundle of degree $0$. By \cite[Proposition 2.10]{Od}, we have $F_C^*F_r\cong F_r$, and hence
$F_{r_i}\in \C_F(X)$. Since $F_{r_i}\otimes L_i$ is essentially finite vector bundles and $\det (F_{r_i}\otimes L_i) \cong L_i$, 
it follows that $L_i$ is essentially finite. By Theorem \ref{LS-thm}(2), we see that $L^{p^n - 1} \cong \struct{X}$ for some $n>0$.
This proves that $E\in \C_F(X)$. 

Now assume that $C$ is super-singular. Let $E$ be an essentially finite vector bundle of rank $r$ on $C$. 
By Lemma \ref{lemma-2}, it follows that the vector bundle $E$ will be of the form \eqref{eq-main}.
If $1\leq r\leq p$, then by \cite[Proposition 2.10]{Od}, we have $F_C^*F_r\cong \struct{C}^r$, and hence 
$F_{r_i}\in \C_F(C)$. Now using the argument as above, we can conclude that $E$ is an object of $\C_F(X)$.
If $r > p$, then again using \cite[Proposition 2.10]{Od}, we can again conclude that $E\in \C_F(C)$. 

\end{proof}

\begin{theorem}\label{thm-EN-Nori}
Let $C$ be an elliptic curve defined over a perfect filed $k$ of positive characteristic $p$. For a $k$-rational point
$x\in C(k)$, we have $\Sfgs{C}{x} \ra \pi^{EN}(C, x) = \Nfgs{C}{x}$.
\end{theorem}
\begin{proof}
First note that if $E\in \C^{EN}(C)$, then $E$ is Nori-semistable. By Lemma \ref{lemma-2}, we have 
$E = \bigoplus _{i = 1}^k F_{r_i}\otimes L_i$, where each $L_i$ is a line bundle of degree $0$. 
Since each $F_r\in \C^\mathrm{uni}(C)\subset \C^N(C)$ and $E\in \C^{EN}(C)$, it follows that each $F_{r_i}\otimes L_i$ 
are semi-essentially finite vector bundles and each $L_i$ is torsion line bundle. This completes the proof.
\end{proof}

\begin{remark}\rm{
In characteristic zero, the Tannakakian group scheme corresponding to the category $\C(F_2)$ generated by 
$S(F_2)$ (in the sense of \cite[\S 2.1]{No}) is isomorphic to $\mathbb{G}_a$ \cite[Proposition 3.3]{Le}. 
While in positive characteristic, we get finite group scheme. To see this, let us recall that 
\begin{equation}\label{eq-multiplication}
F_2\otimes F_r = \left\{ \begin{array}{ll}
         F_r \oplus F_r & \mbox{if $r = 0\mathrm{(mod)} p$};\\
         F_{r-1}\oplus F_{r+1} & \mbox{otherwise}.\end{array} \right.
\end{equation}
Using \eqref{eq-multiplication}, it is easy to see that $S(F_2)$ is finite, and hence $F_2$ is finite vector bundle
\cite[Lemma 3.1]{No}.
This implies that Tannakakian group scheme corresponding to the category $\C(F_2)$ is a finite group scheme.
}
\end{remark}
\subsection*{Acknowledgement} This research work was partially carried out during a visit to ICTS-TIFR Bangalore
for participating in the program - Complex Algerbaic Geometry (Code: ICTS/cag/2018/10).


\end{document}